\documentclass[a4paper,12pt]{article}

\usepackage{color}

\usepackage{makeidx}
\usepackage{bm}
\usepackage{latexsym}
\usepackage{amscd}
\usepackage{amsmath}
\usepackage{amssymb}

\usepackage{stmaryrd}

\usepackage{amsthm}
\usepackage{float}
\usepackage{graphicx}
\usepackage{psfrag}

\theoremstyle{plain}
\newtheorem{theorem}{Theorem}[section]

\newtheorem{lemma}[theorem]{Lemma}

\theoremstyle{definition}
\newtheorem{definition}[theorem]{Definition}
\newtheorem{remark}[theorem]{Remark}

\newtheorem{proposition}[theorem]{Proposition}

\newcommand{\bC}{\ensuremath{\mathbb{C}}}

\newcommand{\bF}{\ensuremath{\mathbb{F}}}

\newcommand{\bP}{\ensuremath{\mathbb{P}}}

\newcommand{\bR}{\ensuremath{\mathbb{R}}}

\newcommand{\bZ}{\ensuremath{\mathbb{Z}}}

\newcommand{\scA}{\ensuremath{\mathcal{A}}}
\newcommand{\scB}{\ensuremath{\mathcal{B}}}
\newcommand{\scC}{\ensuremath{\mathcal{C}}}

\newcommand{\scP}{\ensuremath{\mathcal{P}}}

\newcommand{\Ker}{\operatorname{Ker}}

\newcommand{\A}{\scA}

\newcommand{\R}{\bR}

\newcommand{\barA}{\overline{\scA}}

\newcommand{\barL}{\overline{L}}
\newcommand{\C}{\bC}
\newcommand{\CP}{\bC\bP}
\newcommand{\F}{\bF}
\newcommand{\Z}{\bZ}

\newcommand{\rank}{\operatorname{rank}}

\definecolor{deepblue}{cmyk}{0,0.83,1,0.70}
\definecolor{gray}{cmyk}{0,0,0,0.3}
\definecolor{rred}{cmyk}{0,1,1,0}
\definecolor{chairo}{cmyk}{0,0.83,1,0.70}
\definecolor{roypur}{cmyk}{0.75,0.90,0,0.1}
\definecolor{darkorc}{cmyk}{0.40,0.80,0.20,0}
\definecolor{oliv}{cmyk}{0.64,0.00,0.75,0.56}
\definecolor{azuro}{cmyk}{1,1,0,0.46}

\title{Degeneration of Orlik-Solomon algebras 
and Milnor fibers of complex line arrangements}
\author{Pauline Bailet\thanks{Laboratoire J.A. Dieudonn\'e UMR CNRS 7351 Universit\'e de Nice Sophia-Antipolis Parc Valrose 06108 NICE Cedex 02 FRANCE 
Email: Pauline.BAILET@unice.fr} and 
Masahiko Yoshinaga\thanks{Department of Mathematics, 
Hokkaido University, 
North 10, West 8, Kita-ku, 
Sapporo 060-0810, 
JAPAN 
E-mail: yoshinaga@math.sci.hokudai.ac.jp}}
\date{\today}
\begin{document}
\maketitle

\begin{abstract}
We give a vanishing theorem for the monodromy 
eigenspaces of the Milnor fibers of complex line arrangements. 
By applying the modular bound of the local system cohomology 
groups given by Papadima-Suciu, the result is deduced from 
the vanishing of the cohomology of certain Aomoto complex over finite fields. 
In order to prove this, we introduce degeneration 
homomorphisms of Orlik-Solomon algebras. 

\end{abstract}

\section{Preliminary}
\label{sec:pre}

Let $\barA=\{\overline{L}_0, \overline{L}_1, \dots, \overline{L}_n\}$ be 
an arrangement of $n+1$ lines in the complex projective plane 
$\CP^2$. Let $\alpha_i$ be a defining linear form of $\barL_i$ and 
$Q(x, y, z)=\prod_{i=0}^n\alpha_i$ be the defining 
equation of $\barA$. Then 
$F=F_{\barA}=\{(x, y, z)\in\C^3\mid Q(x, y, z)=1\}\subset\C^3$ 
is called the \emph{Milnor fiber} of 
$\barA$. 
Since $Q$ is homogeneous of degree $n+1$, the Milnor fiber $F$ has 
a natural automorphism $\rho:F\longrightarrow F$, 
$\rho(x, y, z)=e^{2\pi i/(n+1)}\cdot(x, y, z)$, which is called the 
monodromy automorphism. 
The automorphism $\rho$ has order $n+1$. The first cohomology 
group is decomposed into direct sum of eigenspaces 
as $H^1(F, \C)=\bigoplus_{\lambda^{n+1}=1}H^1(F)_\lambda$, where 
the direct sum runs over complex number $\lambda\in\C^\times$ 
with $\lambda^{n+1}=1$. Note that the $1$-eigenspace is always 
$H^1(F)_1\simeq\C^n$. The nontrivial eigenspaces are 
difficult to compute in general (\cite{coh-suc, suc-mil}). 

It has been known that the order of nontrivial 
eigenvalue $\lambda\in\C^\times$ is related to the multiplicities of 
points on a line $\barL_i$ (\cite{lib-mil,cdo}). To be precise, 
we pose the following definition. 

\begin{definition}
Let $k>1$ be an integer larger than $1$. Then 
$$
\mu(\barL_i, k):=
\sharp\{P\in\barL_i\mid \mbox{$k$ divides }\sharp\barA_P\}, 
$$
where $\barA_P=\{\barL\in\barA\mid P\in\barL\}$. In other words, 
$\mu(\barL_i, k)$ is the number of points on $\barL_i$ with multiplicities 
divisible by $k$. 
\end{definition}
In \cite{lib-mil}, Libgober proved the following. 
\begin{theorem}
(Libgober \cite{lib-mil}) 
Let $k>2$ and $\lambda\in\C^\times$ be a nonzero complex number of order $k$. 
If $\mu(\barL_i, k)=0$ for some $\barL_i\in\barA$, 
then $H^1(F)_\lambda=0$. 
\end{theorem}

The purpose of this paper is to generalize the above result 
to the case $\mu(\barL_i, k)=1$. Our first main result is the following. 

\begin{theorem}
\label{thm:main1}
Let $p\in\Z$ be a prime 
and $q=p^\ell$ be its power satisfying $q|(n+1)$. 
Let $\lambda\in\C^\times$ be a nonzero complex number of order $q$. 
Assume that $\barA$ is essential, that is $\barA$ has at least two 
intersections in $\CP^2$. 
If $\mu(\barL_i, p)\leq 1$, for some $\barL_i\in\barA$, then 
$H^1(F)_\lambda=0$. 
\end{theorem}


\begin{remark}
Note that the Milnor fiber $F$ of the 
$A_3$-arrangement ($xyz(x-y)(x-z)(y-z)=0$) 
satisfies $\mu(\barL_i, 3)=2$ for any 
line $\barL_i$, and also $H^1(F)_{e^{2\pi i/3}}\simeq\C$ 
(\cite{coh-suc}). 
Hence nontrivial eigenspaces can appear if 
$\mu(L_i, k)\geq 2$ for any line. 
The assumption 
``$\mu(\barL_i, k)\leq 1$'' can not be weakened. 
\end{remark}

\begin{remark}
In the case that $\barA$ is defined over $\R$, 
a stronger form of Theorem \ref{thm:main1} 
has been proved in \cite{wil, yos-mil}. Namely, 
for any integer $k>1$ and any nonzero complex number 
$\lambda\in\C^\times$ of order $k$, if 
$\mu(\barL_i, k)\leq 1$ for some $\barL_i\in\barA$, then 
$H^1(F)_\lambda=0$. 
\end{remark}

If we choose a line, say $\barL_0$, we can identify 
$\CP^2\setminus\barL_0$ with $\C^2$. 
We obtain the arrangement of affine lines, 
$\A=\{L_1, \dots, L_n\}$ in $\C^2$, 
the so-called deconing of $\barA$, 
where $L_i=\barL_i\cap\C^2$. 
We denote the complement by 
$M(\A)=\C^2\setminus\bigcup_{i=1}^nL_i=\CP^2\setminus
\bigcup_{i=0}^n\barL_i$. 

Let $R$ be a commutative ring. Then the cohomology ring 
$A_R^*(\A):=H^*(M(\A), R)$ is called the Orlik-Solomon algebra, 
which has the following combinatorial description (\cite{orl-ter}). 
\begin{itemize}
\item 
$A_R^0(\A)= R$, 
\item 
$A_R^1(\A)= R\cdot e_1\oplus R\cdot e_2\oplus\dots\oplus R\cdot e_n$, 
\item 
$A_R^2(\A)= A^1_R(\A)^{\wedge 2}/I$, where 
$I\subset A^1_R(\A)^{\wedge 2}$ is an $R$-submodule 
generated by the following elements. 
\begin{itemize}
\item[(i)] $e_i\wedge e_j$, for each pair of parallel lines $L_i//L_j$, 
\item[(ii)] $e_i\wedge e_j-e_i\wedge e_k+e_j\wedge e_k$ for 
three lines such that $L_i\cap L_j\cap L_k\neq\emptyset$. 
\end{itemize}
\item $A^q_R(\A)=0$ for $q\geq 3$. 
\end{itemize}
Let $\xi\in A^1_R(\A)$ be an element of degree one. Then the 
cochain complex $(A_R^\bullet(\A), \xi)$ is called the Aomoto complex. 
Let us denote the sum of generators $e_i$ by 
\begin{equation}
\nu_R:=e_1+e_2+\dots+e_n\in A^1_R(\A). 
\end{equation}
The cohomology of the Aomoto complex 
$(A_R^\bullet(\A), \nu_R)$ 
is known to be related with the eigenspaces of the Milnor 
fiber. When $R$ is a field, we denote 
the rank of the $1$-st cohomology by 
$$
\beta_1(\A, \nu_R):=
\rank_R H^1(A_R^\bullet(\A), \nu_R). 
$$
This invariant gives an upper bound for the 
dimension of eigenspace. 

\begin{theorem}
\label{thm:ps}
(Papadima and Suciu \cite{ps-sp}) 
Let $\barA$ be as before (we do not pose any conditions on the 
multiplicities). Let $q=p^\ell$ be a prime power satisfying 
$q|(n+1)$. Let 
$\lambda\in\C^\times$ be a nonzero complex number of order $q$. 
Then 
\begin{equation}
\dim H^1(F)_\lambda\leq\beta_1(\A, \nu_{\F_p}), 
\end{equation}
where $\F_p=\Z/p\Z$. 
\end{theorem}

Our second main theorem is the following. 

\begin{theorem}
\label{thm:main2}
Let $p\in\Z$ be a prime. 
Assume that $\barA$ is essential, that is $\barA$ has at least two 
intersections in $\CP^2$. 
If $\mu(\barL_i, p)\leq 1$, for some $\barL_i\in\barA$, then 
$H^1(A_{\F_p}^\bullet(\A), \nu_{\F_p})=0$. 
\end{theorem}

Theorem \ref{thm:main1} easily follows from 
Theorem \ref{thm:ps} and Theorem \ref{thm:main2}. 
In order to prove 
Theorem \ref{thm:main2}, we 
introduce degeneration homomorphism 
$A^*_R(\A)\longrightarrow A^*_R(\scB)$ of Orlik-Solomon algebras, 
where $\scB$ is certain ``degenerated arrangement'' in \S\ref{sec:degen}. 
The structure of degenerate Orlik-Solomon algebra $A_R^*(\scB)$ is 
studied in \S\ref{sec:os}. 
By using the degeneration homomorphisms, 
we prove Theorem \ref{thm:main2} in \S\ref{sec:van}. 

\begin{remark}
In this paper, $\barA$ denotes a projective line arrangement 
in $\bC\bP^2$ and $\A$ denotes an affine line arrangement 
in $\bC^2$. 
\end{remark}

\section{Typical Orlik-Solomon algebras}
\label{sec:os}

In this section, we summarize basic results on the first cohomology 
of the Aomoto complex. The results in this section is well know 
(e.g., \cite{lib-yuz}), however, we gave proofs for completeness. 
Let $\A=\{L_1, \dots, L_n\}$ be an arrangement of affine lines 
in $\C^2$. Let $\xi=a_1e_1+\dots+ a_ne_n\in A_R^1(\A)$. 
Let $R$ be a commutative ring. 
Consider 
the Aomoto complex $(A_R^\bullet(\A), \xi)$. 
Define the submodule $A_R^1(\A)_0$ of $A_R^1(\A)$ by 
\begin{equation}
A_R^1(\A)_0:=\{\eta=c_1e_1+\cdots+c_ne_n\in A_R^1(\A)\mid
c_1+c_2+\dots+c_n=0\}. 
\end{equation}
We have the following. 

\begin{lemma}
\label{lem:ker}
If $\sum_{i=1}^n a_i\in R^\times$, then 
\begin{equation}
H^1(A_R^\bullet(\A), \xi)
\simeq
\Ker\left(A_R^1(\A)_0\stackrel{\xi\wedge}{\longrightarrow}
A_R^2(\A)\right). 
\end{equation}
\end{lemma}
\begin{proof}
It is easily seen that 
the natural homomorphism 
$$
\Ker\left(A_R^1(\A)_0\stackrel{\xi\wedge}{\longrightarrow}
A_R^2(\A)\right)\longrightarrow
H^1(A_R^\bullet(\A), \xi)
$$
is surjective and injective. 
\end{proof}

\subsection{Central case}
\label{subsec:cent}

Let $\scC_s=\{L_1, \dots, L_s\}$ be $s$ lines passing through the 
origin as in Figure \ref{fig:central}. 
Let $\xi=a_1e_1+\cdots+a_se_s\in A_R^1(\scC_s)$.

\begin{figure}[htbp]
\begin{picture}(400,100)(0,0)
\thicklines

\put(140,10){\line(1,1){90}}
\put(160,10){\line(2,3){60}}
\put(180,10){\line(1,3){30}}

\multiput(195,4)(10,0){4}{\circle*{3}}

\put(240,10){\line(-2,3){60}}

\put(132,0){$L_1$}
\put(152,0){$L_2$}
\put(172,0){$L_3$}
\put(240,0){$L_s$}

\end{picture}
      \caption{Central arrangement $\scC_s$}
\label{fig:central}
\end{figure}

The module $A_R^1(\scC_s)_0$ is free of rank $s-1$, with a 
basis 
$\{e_i-e_{i+1}\}_{1\leq i\leq s-1}$. 
Also $A_R^2(\scC_s)$ has rank $s-1$ with a basis  
$\{e_i\wedge e_{i+1}\}_{1\leq i\leq s-1}$. 
A straightforward computation shows that 
\begin{equation}
\xi\wedge(e_i-e_{i+1})=
-(a_1+a_2+\cdots+a_s)\cdot e_i\wedge e_{i+1}. 
\end{equation}
Thus we have the following. 

\begin{proposition}
\label{prop:central}
If $\sum_{i=1}^sa_i\in R^\times$, then $H^1(A_R^\bullet(\scC_s), \xi)=0$. 
\end{proposition}

\subsection{Almost parallel case}
\label{subsec:par}

Let $\scP_r=\{L_1, \dots, L_r, L_{r+1}\}$ be an arrangement of 
$r+1$ lines in $\C^2$ such that $L_1, \dots, L_r$ are parallel 
and $L_{r+1}$ is transversal to the others as in 
Figure \ref{fig:parallel}. 
Then $A_R^2(\scP_r)$ has rank $r$ with a basis 
$\{e_i\wedge e_{r+1}\}_{1\leq i\leq r}$. 
Let $\xi=a_1e_1+\cdots+a_re_r+a_{r+1}e_{r+1}\in A_R^1(\scP_r)$.

\begin{figure}[htbp]
\begin{picture}(400,100)(0,0)
\thicklines

\put(140,10){\line(1,3){30}}
\put(160,10){\line(1,3){30}}
\put(210,10){\line(1,3){30}}

\multiput(170,4)(10,0){3}{\circle*{3}}

\put(260,10){\line(-3,2){130}}

\put(130,0){$L_1$}
\put(150,0){$L_2$}
\put(200,0){$L_r$}
\put(260,0){$L_{r+1}$}

\end{picture}
      \caption{Almost parallel arrangement $\scP_r$}
\label{fig:parallel}
\end{figure}

\begin{proposition}
\label{prop:parallel}
If $a_{r+1}\in R^\times$, then 
\begin{equation}
H^1(A_R^\bullet(\scP_r), \xi)\simeq 
\Ker\left(
R\cdot e_1\oplus\cdots\oplus R\cdot e_r
\stackrel{\xi\wedge}{\longrightarrow}
A_R^2(\scP_r)
\right)=0. 
\end{equation}
\end{proposition}

\begin{proof}
By a similar argument to Lemma \ref{lem:ker}, 
we have 
\begin{equation}
H^1(A_R^\bullet(\scP_r), \xi)\simeq 
\Ker\left(
R\cdot e_1\oplus\cdots\oplus R\cdot e_r
\stackrel{\xi\wedge}{\longrightarrow}
A_R^2(\scP_r)
\right). 
\end{equation}
Let $\eta=c_1e_1+\dots +c_re_r$. Then 
$\xi\wedge\eta=-a_{r+1}\sum_{i=1}^r
c_i(e_i\wedge e_{r+1})$. 
Since $a_{r+1}\in R^\times$, $\xi\wedge\eta=0$ implies 
$\eta=0$. 
\end{proof}

\section{Degeneration of OS-algebras}
\label{sec:degen}

Let $\A=\{L_1, \dots, L_n\}$ be an arrangement of 
affine lines in $\C^2$. Consider the decomposition of 
$\A$ into parallel classes, that is a partition 
$$
\A=\A_1\sqcup \A_2\sqcup\cdots\sqcup\A_s, 
$$
such that two lines $L, L'\in\A$ are parallel if and only if 
they are contained in the same class, i.e., 
$L, L'\in\A_\alpha$ for some $1\leq \alpha\leq s$. 
Subsequently, we show the existence of 
certain surjective homomorphisms. The first one is 
\emph{the total degeneration}, 
\begin{equation}
\label{eq:totdeg}
\Delta_{tot}: 
A_R^*(\A)\longrightarrow
A_R^*(\scC_s). 
\end{equation}
Choose a parallel class $\A_\alpha$. 
Let $\sharp\A_\alpha=r$. Then we can construct 
a homomorphism, 
\emph{the directional degeneration} with respect to $\A_\alpha$, 
\begin{equation}
\label{eq:dirdeg}
\Delta_{dir}: 
A_R^*(\A)\longrightarrow
A_R^*(\scP_r). 
\end{equation}

\subsection{Total degeneration}
\label{subsec:tot}

Let $\A=\A_1\sqcup\cdots\sqcup\A_s$ be the parallel 
class decomposition as above. 
Let us denote by $\widetilde{e}_1, \dots, \widetilde{e}_s
\in A_R^1(\scC_s)$ 
the generators of $A_R^*(\scC_s)$. 

\begin{theorem}
\label{thm:totdeg}
There exists an $R$-algebra homomorphism 
$$
\Delta_{tot}: 
A_R^*(\A)\longrightarrow
A_R^*(\scC_s) 
$$
such that 
$$
e_i\longmapsto \widetilde{e}_\alpha
$$
if $L_i\in\A_\alpha$, $1\leq \alpha\leq s$ 
(see Figure \ref{fig:tot} for an example). 
\end{theorem}

\begin{figure}[htbp]
\begin{picture}(400,110)(0,0)
\thicklines

\multiput(50,50)(0,20){2}{\line(1,0){100}}
\multiput(90,10)(20,0){2}{\line(0,1){100}}
\put(55,15){\line(1,1){90}}

\put(35,65){$L_1$}
\put(35,45){$L_2$}
\put(40,5){$L_3$}
\put(85,0){$L_4$}
\put(105,0){$L_5$}

\put(165,40){\vector(1,0){60}}

\put(155,100){\small $\A_1=\{L_1, L_2\}$}
\put(155,85){\small $\A_2=\{L_3\}$}
\put(155,70){\small $\A_3=\{L_4, L_5\}$}

\multiput(250,59)(0,2){2}{\line(1,0){100}}
\multiput(299,10)(2,0){2}{\line(0,1){100}}
\put(255,15){\line(1,1){90}}

\put(235,55){$\widetilde{L}_1$}
\put(240,5){$\widetilde{L}_2$}
\put(300,0){$\widetilde{L}_3$}

\end{picture}
      \caption{Total degeneration of $\A$ to $\scC_3$}
\label{fig:tot}
\end{figure}

\begin{proof}
It is sufficient to show that 
\begin{itemize}
\item[(i)] 
$
\Delta_{tot}(e_i)\wedge
\Delta_{tot}(e_j)=0$ when $L_i, L_j\in\A$ are parallel, and 
\item[(ii)] 
$\Delta_{tot}(e_i)\wedge\Delta_{tot}(e_j)-
\Delta_{tot}(e_i)\wedge\Delta_{tot}(e_k)+
\Delta_{tot}(e_j)\wedge\Delta_{tot}(e_k)=0$ when 
$L_i\cap L_j\cap L_k\neq\emptyset$. 
\end{itemize}
Both are straightforward. Indeed, if $L_i//L_j$, then 
$L_i, L_j\in\A_\alpha$ for some $1\leq\alpha\leq s$. We have 
$\Delta_{tot}(e_i)=\Delta_{tot}(e_j)=\widetilde{e}_\alpha$ 
and (i) holds. 
In the case (ii), we use the relation 
$
\widetilde{e}_\alpha\wedge\widetilde{e}_\beta-
\widetilde{e}_\alpha\wedge\widetilde{e}_\gamma+
\widetilde{e}_\beta\wedge\widetilde{e}_\gamma=0$ for 
any $1\leq\alpha, \beta, \gamma\leq s$. 
\end{proof}

The next Lemma will be useful later. 

\begin{lemma}
\label{lem:kertot}
Let $\xi=\sum_{i=1}^na_ie_i\in A_R^1(\A)$ and 
$\eta\in A_R^1(\A)_0$. 
Assume that $\sum_{i=1}^na_i\in R^\times$ and 
$\xi\wedge\eta=0$. Then $\eta\in\Ker\Delta_{tot}$. 
\end{lemma}
\begin{proof}
Recall that 
$\eta=\sum_{i=1}^nc_ie_i$ is contained in 
$A_R^1(\A)_0$ if and only if 
$\sum_{i=1}^nc_i=0$. By the definition of 
$\Delta_{tot}$, we have 
$\Delta_{tot}(\eta)\in A_R^1(\scC_s)_0$. 

Set $\Delta_{tot}(\xi)=\sum_{\alpha=1}^s
\widetilde{a}_\alpha\widetilde{e}_\alpha$. Then 
$\sum_{\alpha=1}^s\widetilde{a}_\alpha=
\sum_{i=1}^na_i\in R^\times$. Since $\Delta_{tot}$ is 
an algebra homomorphism, we have 
$$
\Delta_{tot}(\xi)\wedge
\Delta_{tot}(\eta)=
\Delta_{tot}(\xi\wedge\eta)=0. 
$$
By Proposition \ref{prop:central} and Lemma \ref{lem:ker}, 
$\Delta_{tot}(\eta)=0$. 
\end{proof}

\subsection{Directional degeneration}
\label{subsec:dir}

As in the previous subsection, let $\A=\A_1\sqcup\cdots\sqcup\A_s$ be 
the parallel class decomposition. Here we fix a parallel class 
$\A_\alpha$. 
Suppose $\sharp\A_\alpha=r$ and 
$\A_\alpha=\{L_{i_1}, \dots, L_{i_r}\}$.

Let us denote by $\widetilde{e}_1, \dots, \widetilde{e}_r, 
\widetilde{e}_{r+1}
\in A_R^1(\scP_r)$ 
the generators of $A_R^*(\scP_r)$ such that 
$\widetilde{e}_1, \dots, \widetilde{e}_r$ correspond to 
parallel lines and 
$\widetilde{e}_{r+1}$ corresponds to the transversal line. 
(See \S\ref{subsec:par}.) 

\begin{theorem}
\label{thm:dirdeg}
Under the above setting, 
there exists a homomorphism of $R$-algebras 
$$
\Delta_{dir}: 
A_R^*(\A)\longrightarrow
A_R^*(\scP_r), 
$$
such that 
$$
\Delta_{dir}(e_i)=
\left\{
\begin{array}{ll}
\widetilde{e}_u&\mbox{ if $L_i\in\A_\alpha$ and $i=i_u$ $(1\leq u\leq r)$,}\\
\widetilde{e}_{r+1}&\mbox{ if $L_i\notin\A_\alpha$.}
\end{array}
\right.
$$
(see Figure \ref{fig:dir} for an example.) 
\end{theorem}

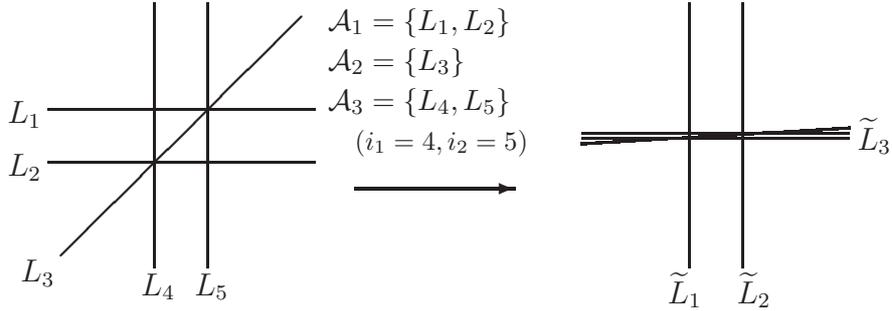
\begin{figure}[htbp]
\begin{picture}(400,110)(0,0)
\thicklines

\multiput(50,50)(0,20){2}{\line(1,0){100}}
\multiput(90,10)(20,0){2}{\line(0,1){100}}
\put(55,15){\line(1,1){90}}

\put(35,65){$L_1$}
\put(35,45){$L_2$}
\put(40,5){$L_3$}
\put(85,0){$L_4$}
\put(105,0){$L_5$}

\put(165,40){\vector(1,0){60}}

\put(155,100){\small $\A_1=\{L_1, L_2\}$}
\put(155,85){\small $\A_2=\{L_3\}$}
\put(155,70){\small $\A_3=\{L_4, L_5\}$}
\put(165,55){\footnotesize $(i_1=4, i_2=5)$}

\multiput(250,59)(0,2){2}{\line(1,0){100}}
\qbezier(250,57)(300,60)(350,63)
\multiput(290,10)(20,0){2}{\line(0,1){100}}

\put(353,55){$\widetilde{L}_3$}
\put(282,-3){$\widetilde{L}_1$}
\put(308,-3){$\widetilde{L}_2$}

\end{picture}
      \caption{Directional degeneration of $\A$ to $\scP_2$ 
with respect to $\A_3$}
\label{fig:dir}
\end{figure}

\begin{proof}
Similar to the proof of Theorem \ref{thm:totdeg}. 
\end{proof}

%

\section{Vanishing Results}
\label{sec:van}

Now we prove Theorem \ref{thm:main2}. Without loss of generality, 
we may assume $\mu(\overline{L}_0, p)\leq 1$ (otherwise, we may change 
the line at infinity so that $\mu(\overline{L}_0, p)\leq 1$ is satisfied).

Let $\A=\A_1\sqcup\cdots\sqcup\A_s$ be the parallel class decomposition. 
Each class $\A_\alpha$ of parallel lines determines an intersection 
$P_\alpha\in\overline{L}_0$ on the line at infinity. 
Let us define $m_\alpha$ by $m_\alpha:=\sharp(\barA_{P_\alpha})-1$. 
By definition, the multiplicity at $P_\alpha$ is equal to $m_\alpha+1$. 
We may assume that $P_1\in \barL_0$ is the unique point on $\barL_0$ 
with the multiplicity divisible by the prime $p$. In other words, 
we may assume 
$p|(m_1+1)$ and $p\not |(m_\alpha+1)$ for $\alpha\geq 2$. 

We first note that since $p|(n+1)$, we have $n=-1\in(\F_p)^\times$ (via the 
natural homomorphism $\Z\longrightarrow\F_p$). 
This enables us to apply Lemma \ref{lem:ker} for $\xi=\nu_{\F_p}$. 
Let 
$$
\eta=\sum_{i=1}^n c_ie_i\in A_{\F_p}^1(\A)_0
$$
and 
assume that $\eta\wedge\nu_{\F_p}=0$. We shall prove that $\eta=0$. 
The idea is as follows. We first consider total degeneration of 
$\eta$, and find certain relations among coefficients. Then we apply 
the directional degeneration with respect to $\A_1$. We will prove 
that if $\eta\neq 0$, then $\eta\wedge\nu_{\F_p}\neq0$, which 
contradicts the assumption. 

Consider first 
the total degeneration $\Delta_{tot}:A_{\F_p}^*(\A)\longrightarrow
A_{\F_p}^*(\scC_s)$. 
From the definition of the map $\Delta_{tot}$, 
$$
\Delta_{tot}(\eta)=
\sum_{\alpha=1}^s
\left(
\sum_{L_i\in\A_\alpha}c_i
\right)
\cdot\widetilde{e}_{\alpha}
\in A_{\F_p}^1(\scC_s)_0. 
$$
By Lemma \ref{lem:kertot}, $\eta\in\Ker\Delta_{tot}$, and hence we have 
\begin{equation}
\label{eq:zero}
\sum_{L_i\in\A_\alpha}c_i=0, \mbox{ for each }
1\leq\alpha\leq s. 
\end{equation}
Next we consider the directional degeneration 
$\Delta_{dir}:A_{\F_p}^*(\A)\longrightarrow A_{\F_p}^*(\scP_{m_\alpha})$ 
with respect to $\A_\alpha$. 
Suppose that 
$\A_\alpha=\{L_{i_1}, \dots, L_{i_{m_\alpha}}\}$. 
Then 
by definition, $\Delta_{dir}(e_{i_j})=\widetilde{e}_j$ for 
$1\leq j\leq m_\alpha$ and $\Delta_{dir}(e_k)=\widetilde{e}_{m_\alpha+1}$ 
for $L_k\notin\A_\alpha$. Thus we have 
\begin{equation}
\label{eq:nu}
\Delta_{dir}(\nu_{\F_p})=
\widetilde{e}_1+\dots+\widetilde{e}_{m_\alpha}+
(n-m_\alpha)\widetilde{e}_{m_\alpha+1}. 
\end{equation}
On the other hand, by using (\ref{eq:zero}), we have 
\begin{equation}
\label{eq:eta}
\begin{split}
\Delta_{dir}(\eta)&=
c_{i_1}\widetilde{e}_1+\dots+
c_{i_{m_\alpha}}\widetilde{e}_{m_\alpha}+
\left(
\sum_{\beta\neq\alpha}
\left(
\sum_{L_k\in\A_\beta}c_k
\right)
\right)\widetilde{e}_{m_\alpha+1}
\\
&=c_{i_1}\widetilde{e}_1+\dots+
c_{i_{m_\alpha}}\widetilde{e}_{m_\alpha}. 
\end{split}
\end{equation}

Note that 
$n-m_\alpha=(n+1)-(m_\alpha+1)$. 
If $\alpha\geq 2$, then 
$m_\alpha+1=\sharp\barA_{P_\alpha}$ is not divisible by $p$. 
Hence we have $n-m_\alpha\in\F_p^\times$. By the assumption, 
we also have $n-m_1=0\in\F_p$. 
Using 
(\ref{eq:nu}) and (\ref{eq:eta}), we can apply 
Proposition \ref{prop:parallel}, and we have 
$\Delta_{dir}(\eta)=0$ (equivalently, $c_i=0$ for $L_i\in\A_\alpha$) for 
any $\alpha\neq 1$. 
Now, for simplicity, we set $\A_1=\{L_1, L_2, \dots, L_r\}$ ($r=m_1$). 
From the assumption that $\barA$ is essential, $r<n$. 
We may assume that $\eta$ has an expression 
\begin{equation}
\label{eq:expresseta}
\eta=
c_1e_1+\cdots+
c_re_r, 
\end{equation}
with $c_1+c_2+\dots +c_r=0$. 
By using Brieskorn decomposition (\cite{orl-ter}) 
$A_R^2(\A)=\bigoplus_{X\in L_2(\A)}A_R^2(\A_X)$, ($X$ runs all 
codimension two intersections, i.e., points), it is easily seen 
that 
\begin{equation}
\{e_i\wedge e_j\mid 1\leq i\leq r, r+1\leq j\leq n\}
\end{equation}
is a linearly independent set in $A_{\F_p}^2(\A)$. Hence if 
$\eta\neq 0$, then we have 
$$
\eta\wedge\nu_{\F_p}=\eta\wedge(e_{r+1}+\dots+e_n)\neq 0. 
$$
(Note that the equation (\ref{eq:expresseta}) implies that 
the support of $\eta$ is contained in the parallel 
class $\A_1$.) 
Thus $\eta\wedge\nu_{\F_p}=0$ implies that $\eta=0$. 
This completes the proof of Theorem \ref{thm:main2}.



\medskip

\noindent
{\bf Acknowledgement.} 
The authors thank 
Professor Alexandru Dimca for his 
suggestions and encouragements. 
The second author is supported by 
JSPS Grant-in-Aid for Scientific Research (C).

\end{document}